\documentclass[10pt,twoside]{siamart1116}

\usepackage[english]{babel}
\usepackage{graphicx,epstopdf,epsfig}
\usepackage{amsfonts,epsfig,fancyhdr,graphics, hyperref,amsmath,amssymb}

\newcommand{\HE}{Name of Handling Editor}
\newcommand{\DoS}{Month/Day/Year}
\newcommand{\DoA}{Month/Day/Year}
\newcommand{\CA}{Name of Corresponding Author}

\newcommand{\Names}{D.B. Janse van Rensburg, A.C.M. Ran, F. Theron and M. van Straaten}
\newcommand{\Title}{$m$th roots of $H$-selfadjoint matrices over the quaternions}

\renewcommand{\theequation}{\thesection.\arabic{equation}}
\renewtheorem{theorem}{Theorem}[section]

\newcommand{\I}{\mathsf{i}}
\newcommand{\J}{\mathsf{j}}
\newcommand{\K}{\mathsf{k}}

\begin{document}

\bibliographystyle{plain}

\setcounter{page}{1}

\thispagestyle{empty}

 \title{\Title\thanks{Received
 by the editors on \DoS.
 Accepted for publication on \DoA. 
 Handling Editor: \HE. Corresponding Author: \CA}}

\author{
D.B.~Janse~van~Rensburg\thanks{Department of Mathematics and Applied Mathematics,
Research Focus: Pure and Applied Analytics, North-West~University,
Private~Bag~X6001,
Potchefstroom~2520,
South Africa.
(dawie.jansevanrensburg@nwu.ac.za, frieda.theron@nwu.ac.za, madelein.vanstraaten@nwu.ac.za). Supported by a grant from DSI-NRF Centre of Excellence in Mathematical and Statistical Sciences (CoE-MaSS).}
\and
A.C.M.~Ran\thanks{Department of Mathematics, Faculty of Science, VU University Amsterdam, De Boelelaan
    1111, 1081 HV Amsterdam, The Netherlands
    and Research Focus: Pure and Applied Analytics, North-West~University,
Potchefstroom,
South Africa. (a.c.m.ran@vu.nl).}
\and 
F.~Theron\footnotemark[2]
\and
M.~van~Straaten\footnotemark[2]}

\markboth{\Names}{\Title}

\maketitle

\begin{abstract}
The complex matrix representation for a quaternion matrix is used in this paper to find necessary and sufficient conditions for the existence of an $H$-selfadjoint $m$th root of a given $H$-selfadjoint quaternion matrix. In the process, when such an $H$-selfadjoint $m$th root exists, its construction is also given.
\end{abstract}



\section{Introduction}
Denote the skew-field of real quaternions by $\mathbb{H}$. The basic theory of quaternion linear algebra can be found in various books and papers, see for example the book by Rodman, \cite{Rodman}, and \cite{Zhang,ZhangWei}. Let $m$ be any positive integer and let $H$ be a square quaternion matrix which is invertible and Hermitian. We focus on the class of selfadjoint matrices relative to the indefinite inner product generated by  $H$. In the general sense of the definition a square quaternion matrix $A$ is said to be $H$-selfadjoint if $HA=A^*H$.

If $B$ is a square $H$-selfadjoint matrix with quaternion entries, we seek to find necessary and sufficient conditions for the existence of an $H$-selfadjoint matrix $A$ such that $A^m=B$. This matrix $A$ is referred to as an $H$-selfadjoint $m$th root of the matrix $B$.

It is well known that there exists a complex matrix representation for a matrix with quaternion entries, that is, there exists an isomorphism $\omega_n$ between the real algebra of all $n\times n$ quaternion matrices and a subalgebra $\Omega_{2n}$ of the algebra of all $2n\times 2n$ complex matrices. This isomorphism $\omega_n$ maps an $n\times n$ quaternion matrix $A=A_1+\J A_2$, where $A_1,A_2\in\mathbb{C}^{n\times n}$, to a $2n\times 2n$ matrix
$$\begin{bmatrix}
A_1 & \bar{A}_2 \\
-A_2 & \bar{A}_1
\end{bmatrix}.$$
  This fact simplifies our problem. Therefore we find necessary and sufficient conditions for the existence of an $\tilde{H}$-selfadjoint $m$th root in $\Omega_{2n}$ of an $\tilde{H}$-selfadjoint matrix in $\Omega_{2n}$. 
  
Quaternion matrices in indefinite inner product spaces, and specifically different canonical forms, were studied in \cite{Alpay,Djokovic,Karow,Sergeichuk}. Not much research has been done on roots of quaternion matrices, although \cite{Jafari} does give a formula for obtaining $m$th roots of quaternion matrices of a particular form. On the other hand, in the complex case $H$-selfadjoint $m$th roots of $H$-selfadjoint matrices have been studied extensively. Necessary and sufficient conditions for the existence of such roots can be found in \cite{mthroots}. For the case of $H$-selfadjoint square roots and applications to polar decompositions of $H$-selfadjoint matrices, see \cite{BMRRR}, and for the case of square roots of $H$-nonnegative matrices, see \cite{CT}. We refer to the introduction of a previous paper \cite{mthroots}, for an overview of $m$th roots of matrices in general.

Note that in the process of finding an $H$-selfadjoint $m$th root $A\in\Omega_{2n}$ of $B\in\Omega_{2n}$, i.e.\ $A^m=B$, the functional analytic approach via the Cauchy integral
\[A:=\frac{1}{2\pi \I}\int_\Gamma \sqrt[m]{\lambda}(\lambda I-B)^{-1}\,d\lambda,\]
which can be found for example in \cite{Higham}, would be sufficient in the case where there are no eigenvalues of $B$ in $(-\infty, 0]$. However, in this paper we prefer to take a more direct approach also in the case where the eigenvalues do not lie on the negative real line.

\section{Preliminaries}
We recap the basic theory for quaternions and matrices with quaternion entries as found in the book by Leiba Rodman, \cite{Rodman}. Every element in $\mathbb{H}$ is of the form
\begin{equation*}
x=x_0+x_1\I+x_2\J+x_3\K,
\end{equation*}
where $x_0,x_1,x_2,x_3\in\mathbb{R}$ and the elements $\I,\J,\K$ satisfy the following formulas \begin{equation*}
\I^2=\J^2=\K^2=-1,\quad\I\J=-\J\I=\K,\quad\J\K=-\K\J=\I,\quad\K\I=-\I\K=\J.
\end{equation*}
 It is important to keep in mind the fact that multiplication in $\mathbb{H}$ is not commutative, i.e.\ in general $xy\neq yx$ for $x,y\in\mathbb{H}$. Let $\bar{x}=x_0-x_1\I-x_2\J-x_3\K$ denote the conjugate quaternion of $x$. For a quaternion matrix $A$, let $\bar{A}$ denote the matrix in which each entry is the conjugate of the corresponding entry in $A$.

Let $A$ be an $n\times n$ quaternion matrix, i.e.\ $A\in\mathbb{H}^{n\times n}$, and write $A$ as $A=A_1+\J A_2$ where $A_1,A_2\in\mathbb{C}^{n\times n}$. The map $\omega_n:\mathbb{H}^{n\times n}\rightarrow\mathbb{C}^{2n\times 2n}$ is defined by
\begin{equation*}\label{eqomega(A)}
\omega_n(A)=\begin{bmatrix}
A_1 & \bar{A}_2 \\
-A_2 & \bar{A}_1
\end{bmatrix}.
\end{equation*}
Then $\omega_n$ is an isomorphism of the real algebra $\mathbb{H}^{n\times n}$ onto the real unital subalgebra
\begin{equation*}
\Omega_{2n}:=\left\{\begin{bmatrix}
A_1 & \bar{A}_2 \\
-A_2 & \bar{A}_1
\end{bmatrix} \mid A_1,A_2\in\mathbb{C}^{n\times n}\right\}
\end{equation*}
of $\mathbb{C}^{2n\times 2n}$. The following properties can be found in \cite[Section~3.4]{Rodman} and \cite{Zhang}. Let $X,Y\in\mathbb{H}^{n\times n}$ and $s,t\in\mathbb{R}$ be arbitrary, then
\begin{enumerate}
\item[(i)] $\omega_n(I_n)=I_{2n}$;
\item[(ii)] $\omega_n(XY)=\omega_n(X)\omega_n(Y)$;
\item[(iii)] $\omega_n(sX+tY)=s\omega_n(X)+t\omega_n(Y)$;
\item[(iv)] $\omega_n(X^*)=(\omega_n(X))^*$;
\item[(v)] $\omega_n(X^{-1})=(\omega_n(X))^{-1}$ if $X$ is invertible.
\end{enumerate}

Note that the matrix $X^*\in\mathbb{H}^{n\times m}$ is obtained from $X\in\mathbb{H}^{m\times n}$ by replacing each entry with its conjugate quaternion and then taking the transpose. This isomorphism between $\mathbb{H}^{n\times n}$ and $\Omega_{2n}$ ensures that results for matrices in $\mathbb{H}^{n\times n}$ that are purely algebraic, are also true for matrices in the subalgebra $\Omega_{2n}$ since we can apply $\omega_n$, and vice versa as long as we stay within the subalgebra $\Omega_{2n}$. All definitions that follow could also be made with respect to matrices in  $\Omega_{2n}$.

An $n \times n$ quaternion matrix $A$ has left eigenvalues and right eigenvalues but since we only work with right eigenvalues and right eigenvectors, we will refer to them as \textit{eigenvalues} and \textit{eigenvectors}. 

\begin{definition}
A nonzero vector $v\in\mathbb{H}^n$ is called an \emph{eigenvector} of a matrix $A\in\mathbb{H}^{n\times n}$ corresponding to the \emph{eigenvalue} $\lambda\in\mathbb{H}$ if the equality $Av= v\lambda$ holds.
\end{definition}
The \textit{spectrum} of $A$, denoted by $\sigma(A)$, is the set of all eigenvalues of $A$. Note that $\sigma(A)$ is closed under similarity of quaternions, i.e., if $v$ is an eigenvector of $A$ corresponding to the eigenvalue $\lambda$, then $v\alpha$ is an eigenvector of $A$ corresponding to the eigenvalue $\alpha^{-1}\lambda\alpha$, for all nonzero $\alpha\in\mathbb{H}$. From \cite{Zhang} we see that $A$ has exactly $n$ eigenvalues which are complex numbers with nonnegative imaginary parts and the Jordan normal form of $A$ has exactly these numbers on the diagonal. Let $\mathbb{C}_+=\{a+\I b\mid a\in\mathbb{R},b>0\}$.

Let a single Jordan block of size $n\times n$ at the eigenvalue $\lambda$ be denoted by $J_n(\lambda)$. The $n\times n$ matrix with ones on the main anti-diagonal and zeros elsewhere,  called a standard involutary permutation (sip) matrix, is denoted by $Q_n$.

Recall the following definition from \cite{Shapiro}.
\begin{definition} 
Let $A$ be a square quaternion matrix with Jordan blocks $\bigoplus_{i=1}^r J_{n_i}(\lambda)$ at the eigenvalue $\lambda$ in its Jordan normal form  and assume that $n_1\geq n_2\geq n_3\geq\ldots\geq n_r>0$. The \emph{Segre characteristic} of $A$ corresponding to the eigenvalue $\lambda$ is defined as the sequence
\[n_1,n_2,n_3,\ldots,n_r,0,0,\ldots .\]
\end{definition}

We will use this definition mostly in the case where $\lambda$ is equal to zero unless indicated otherwise and therefore sometimes will simply use   \textit{Segre characteristic} to refer to the Segre characteristic of $A$ corresponding to the eigenvalue $0$. 
Note that the Jordan normal form of matrices in the subalgebra $\Omega_{2n}$ can be found from the Jordan normal form of matrices in $\mathbb{H}^{n\times n}$ since $\omega_n(J_n(\lambda))=J_n(\lambda)\oplus J_n(\bar{\lambda})$, for $\lambda\in\mathbb{C}$. Then it is easy to see the following result.

\begin{corollary}\label{CorDoubleNrinSegre}
If a nilpotent matrix $A$ is in $\Omega_{2n}$, then each number in the Segre characteristic of $A$ occurs twice. 
\end{corollary}

This actually holds for the Segre characteristic corresponding to any real eigenvalue in the case where $A\in\Omega_{2n}$ has real numbers in its spectrum. However, only the nilpotent case will be used later.

A matrix $X\in\mathbb{H}^{n\times n}$ is said to be \textit{Hermitian} if $X^*=X$. Let $H\in\mathbb{H}^{n\times n}$ be an invertible Hermitian matrix. We consider the indefinite inner product $[\cdot,\cdot]$ generated by $H$:
\begin{equation*}
[x,y]=\langle Hx,y\rangle=y^*Hx,\quad x,y\in\mathbb{H}^n,
\end{equation*}
where $\langle\cdot,\cdot\rangle$ denotes the standard inner product. A matrix $A\in\mathbb{H}^{n\times n}$ is called \textit{$H$-selfadjoint} if $HA=A^*H$. Two pairs of matrices $(A_1,H_1)$ and $(A_2,H_2)$ are said to be \textit{unitarily similar} if there exists an invertible quaternion matrix $S$ such that $S^{-1}A_1S=A_2$ and $S^*H_1S=H_2$ hold.

As in the complex case, there exists a canonical form for the pair $(A,H)$ where $A\in\mathbb{H}^{n\times n}$ is an $H$-selfadjoint matrix. This is given in, for example \cite[Theorem~4.1]{Alpay},  \cite{Karow} and \cite[Theorem~10.1.1]{Rodman}.
\begin{theorem}\label{ThmcanonformH}
Let $H\in\mathbb{H}^{n\times n}$ be an invertible Hermitian matrix and $A\in\mathbb{H}^{n\times n}$ an $H$-selfadjoint matrix. Then there exists an invertible matrix $S\in\mathbb{H}^{n\times n}$ such that
\begin{eqnarray}\label{eqcanonformH1}
 S^{-1}AS&=&J_{k_1}(\lambda_1)\oplus\cdots\oplus J_{k_\alpha}(\lambda_\alpha)\nonumber\\
 &\oplus & \begin{bmatrix}
J_{k_{\alpha+1}}(\lambda_{\alpha+1}) & 0 \\ 0 & J_{k_{\alpha+1}}(\overline{\lambda}_{\alpha+1}) 
 \end{bmatrix}\oplus\cdots\oplus
\begin{bmatrix}
J_{k_{\beta}}(\lambda_{\beta}) & 0 \\ 0 & J_{k_{\beta}}(\overline{\lambda}_{\beta}) 
 \end{bmatrix},
\end{eqnarray}
where $\lambda_i\in\sigma(A)\cap\mathbb{R}$ for all $i=1,\ldots,\alpha$, $\lambda_i\in\sigma(A)\cap\mathbb{C}_+$ for all $i=\alpha+1,\ldots,\beta$, and
\begin{equation}\label{eqcanonformH2}
 S^*HS=\eta_1Q_{k_1}\oplus\cdots\oplus \eta_\alpha Q_{k_\alpha}\oplus Q_{2k_{\alpha+1}}\oplus\cdots\oplus Q_{2k_\beta} ,
\end{equation}
where $\eta_i=\pm 1$. The form $(S^{-1}AS,\,S^*HS)$ in \eqref{eqcanonformH1} and \eqref{eqcanonformH2} is uniquely determined by the pair $(A,H)$, up to a permutation of diagonal blocks.
\end{theorem}

We refer to the pair $(S^{-1}AS,S^*HS)$ in \eqref{eqcanonformH1} and \eqref{eqcanonformH2} as the \textit{canonical form} of the pair $(A,H)$ of quaternion matrices.

The following result from \cite{mthroots}  holds for quaternion matrices as well,  but by applying $\omega_n$ it also holds for matrices in the subalgebra $\Omega_{2n}$. We use it in the proofs throughout this paper.
\begin{lemma}\label{Lem2.1recall}
Let $X$ and $Y$ be $n\times n$ quaternion matrices such that the pair $(X,H_X)$ is unitarily similar to the pair $(Y,H_Y)$ where $H_X\in\mathbb{H}^{n\times n}$ and $H_Y\in\mathbb{H}^{n\times n}$ are invertible Hermitian matrices, i.e.\ there exists an invertible matrix $P\in\mathbb{H}^{n\times n}$ such that
\begin{equation*}
P^{-1}XP=Y\quad\textit{and}\quad P^*H_XP=H_Y.
\end{equation*}
Let the matrix $J\in\mathbb{H}^{n\times n}$ be an $H_X$-selfadjoint $m$th root of $X$, i.e.\ $J^m=X$. Then the matrix $A:=P^{-1}JP$ is an $H_Y$-selfadjoint $m$th root of $Y$.
\end{lemma}

From the properties of the map $\omega_n$ we see that a matrix $A$ is invertible if and only if $\omega_n(A)$ is invertible. Also, from Proposition~3.4.1 in \cite{Rodman}  $A$ is Hermitian if only if $\omega_n(A)$ is Hermitian.

To find a canonical form for a pair $(\tilde{A},\tilde{H})$ where $\tilde{A}\in\Omega_{2n}$ is $\tilde{H}$-selfadjoint, and $\tilde{H}\in\Omega_{2n}$ is invertible and Hermitian, we apply $\omega_n$ to the equations \eqref{eqcanonformH1} and \eqref{eqcanonformH2}. Denote the right-hand side of \eqref{eqcanonformH1} by $J$ and the right-hand side of \eqref{eqcanonformH2} by $Q$ and then we have
\begin{equation*}
(\omega_n(S))^{-1}\omega_n(A)\omega_n(S)=\omega_n(J),\quad \omega_n(S)^*\omega_n(H)\omega_n(S)=\omega_n(Q).
\end{equation*}
The uniqueness follows from the fact that the canonical form of $(A,H)$ is  unique (Theorem~\ref{ThmcanonformH}) and that $\omega_n$ is an isomorphism.
Therefore the canonical form of a pair of matrices in $\Omega_{2n}$ is as follows.

\begin{theorem}\label{ThmcanonformOmega}
Let $\tilde{H}\in\Omega_{2n}$ be an invertible Hermitian matrix and $\tilde{A}\in\Omega_{2n}$ an $\tilde{H}$-selfadjoint matrix. Then there exists an invertible matrix $S\in\Omega_{2n}$ such that
\begin{eqnarray}\label{eqcanonformOmega1}
S^{-1}\tilde{A}S&=&J_{k_1}(\lambda_1)\oplus\cdots\oplus J_{k_\alpha}(\lambda_\alpha)\oplus \begin{bmatrix}
J_{k_{\alpha+i}}(\lambda_{\alpha+i}) &0 \\ 0& J_{k_{\alpha+i}}(\overline{\lambda}_{\alpha+i}) 
\end{bmatrix}\oplus\cdots\oplus\begin{bmatrix}
J_{k_\beta}(\lambda_\beta) &0 \\ 0& J_{k_\beta}(\overline{\lambda}_\beta) 
\end{bmatrix} \nonumber\\
&\oplus& J_{k_1}(\lambda_1)\oplus\cdots\oplus J_{k_\alpha}(\lambda_\alpha)\oplus\begin{bmatrix}
J_{k_{\alpha+i}}(\overline{\lambda}_{\alpha+i}) &0 \\ 0& J_{k_{\alpha+i}}(\lambda_{\alpha+i}) 
\end{bmatrix}\oplus\cdots\oplus\begin{bmatrix}
J_{k_\beta}(\overline{\lambda}_\beta) &0 \\ 0& J_{k_\beta}(\lambda_\beta)
\end{bmatrix},
\end{eqnarray}
where $\lambda_i\in\sigma(\tilde{A})\cap\mathbb{R}$ for all $i=1,\ldots,\alpha$, $\lambda_i\in\sigma(\tilde{A})\cap\mathbb{C}_+$ for all $i=\alpha+1,\ldots,\beta$; and
\begin{eqnarray}\label{eqcanonformOmega2}
S^*\tilde{H}S &=& \eta_1
Q_{k_1}\oplus\cdots\oplus\eta_\alpha Q_{k_\alpha}
\oplus
Q_{2k_{\alpha+1}} \oplus\cdots\oplus Q_{2k_\beta}\nonumber\\
&\oplus & \eta_1
Q_{k_1}\oplus\cdots\oplus\eta_\alpha Q_{k_\alpha}
\oplus
Q_{2k_{\alpha+1}} \oplus\cdots\oplus Q_{2k_\beta},
\end{eqnarray}
where $\eta_i=\pm 1$. The form $(S^{-1}\tilde{A}S,S^*\tilde{H}S)$ in \eqref{eqcanonformOmega1} and \eqref{eqcanonformOmega2} is uniquely determined by the pair $(\tilde{A},\tilde{H})$ up to a permutation of diagonal blocks.
\end{theorem}
We note at this stage that the canonical form of the nonreal part of an $H$-selfadjoint matrix $A$ must be of dimensions a multiple of 4 due to the direct sums of $J_{k_i}(\lambda_{i})$ and $J_{k_i}(\overline{\lambda}_{i})$ occurring twice. 
It is crucial to ensure that all matrices throughout the proofs are in $\Omega_{2n}$ and for this reason we give the following result to explain why we can study different Jordan blocks separately.
\begin{lemma}\label{LemBlockPerBlock}
Let $H_1=Q_1\oplus \bar{Q}_1,B_1=J_1\oplus \bar{J}_1\in\Omega_{2n}$ and $H_2=Q_2\oplus \bar{Q}_2,B_2=J_2\oplus \bar{J}_2\in\Omega_{2p}$ where $B_1$ is $H_1$-selfadjoint and $B_2$ is $H_2$-selfadjoint. Let $A_1\in\Omega_{2n}$ be an $H_1$-selfadjoint $m$th root of $B_1$ and $A_2\in\Omega_{2p}$ an $H_2$-selfadjoint $m$th root of $B_2$, and let their entries be as follows
\begin{equation*}
A_1=\begin{bmatrix}
A_{11}^{(1)} & \bar{A}_{12}^{(1)}\\
-A_{12}^{(1)} & \bar{A}_{11}^{(1)}
\end{bmatrix}\quad \textit{and}\quad A_2=\begin{bmatrix}
A_{11}^{(2)} & \bar{A}_{12}^{(2)}\\
-A_{12}^{(2)} & \bar{A}_{11}^{(2)}
\end{bmatrix}.
\end{equation*}
Let $\hat{B}=J_1\oplus J_2\oplus \bar{J}_1\oplus \bar{J}_2\in\Omega_{2(n+p)}$ and $\hat{H}=Q_1\oplus Q_2\oplus \bar{Q}_1\oplus \bar{Q}_2\in\Omega_{2(n+p)}$. Then $\hat{B}$ is $\hat{H}$-selfadjoint and the matrix 
\begin{equation*}
\hat{A}=\begin{bmatrix}
A_{11}^{(1)} &0& \bar{A}_{12}^{(1)}&0\\
0& A_{11}^{(2)} &0& \bar{A}_{12}^{(2)}\\
-A_{12}^{(1)} &0& \bar{A}_{11}^{(1)}&0 \\
0& -A_{12}^{(2)} &0& \bar{A}_{11}^{(2)}
\end{bmatrix}\in\Omega_{2(n+p)}
\end{equation*}
is an $\hat{H}$-selfadjoint $m$th root of the matrix $\hat{B}$.
\end{lemma}

\begin{proof}
Let $P$ be the following permutation matrix
\begin{equation*}
P=\begin{bmatrix}
I_n &0&0&0\\
0&0&I_p&0\\
0&I_n&0&0\\
0&0&0&I_p
\end{bmatrix}.
\end{equation*} 
This $P$ produces a map from $\Omega_{2n}\oplus\Omega_{2p}$ to $\Omega_{2(n+p)}$ which satisfies $P(A_1\oplus A_2)P^{-1}=\hat{A}$. We also then have $P(B_1\oplus B_2)P^{-1}=\hat{B}$ and $P(H_1\oplus H_2)P^{-1}=\hat{H}$. Therefore
\begin{equation*}
\hat{A}^m=\left(P(A_1\oplus A_2)P^{-1}\right)^m=P(A_1\oplus A_2)^mP^{-1}=P(B_1\oplus B_2)P^{-1}=\hat{B}.
\end{equation*}
Now, by noting that $P^*=P^{-1}$ and using the facts that $B_1$ and $A_1$ are $H_1$-selfadjoint and $B_2$ and $A_2$ are $H_2$-selfadjoint, it follows that $\hat{H}\hat{B}=\hat{B}^*\hat{H}$ and $\hat{H}\hat{A}=\hat{A}^*\hat{H}$.
Therefore $\hat{A}$ is an $\hat{H}$-selfadjoint $m$th root of the $\hat{H}$-selfadjoint matrix $\hat{B}$.
\end{proof}

The matrix $J_i$ in Lemma~\ref{LemBlockPerBlock} can be the Jordan normal form $J_{k}(\lambda)$ in the case where $\lambda$ is real or the Jordan normal form $J_k(\lambda)\oplus J_k(\bar{\lambda})$ in the case where $\lambda$ is nonreal.

In general, let the permutation matrix $P$ be a $2t\times 2t$ block matrix where the block in the $i$th row and $j$th column is defined by
\begin{equation}\label{eqPermutationMatrix}
P_{ij}=\begin{cases}
I_{2k_i} & \text{if } j=2i-1,\,i\leq t; \\
I_{2k_{i-t}} & \text{if }j=2(i-t),\, i>t;\\
0 & \text{otherwise} . 
\end{cases}
\end{equation}
Then we have, for example, that
\begin{equation*}
P\bigoplus_{j=1}^t\left(Q_{k_j}\oplus -Q_{k_j}\oplus Q_{k_j}\oplus -Q_{k_j}\right)P^{-1}=\bigoplus_{j=1}^t\left(Q_{k_j}\oplus -Q_{k_j}\right)\oplus \bigoplus_{j=1}^t\left(Q_{k_j}\oplus -Q_{k_j}\right).
\end{equation*}

\section{Existence of $m$th roots}

We first present a very handy tool for working with quaternion matrices.
\begin{lemma}
Let $H$ be an $n\times n$ quaternion matrix which is invertible and Hermitian,  and let $B$ be an $n\times n$ $H$-selfadjoint quaternion matrix. There exists an $H$-selfadjoint quaternion matrix $A$ such that $A^m=B$ if and only if there exists an $\tilde{H}=\omega_n(H)$-selfadjoint matrix $\tilde{A}=\omega_n(A)$ such that $\tilde{A}^m=\tilde{B}$, where $\tilde{B}=\omega_n(B)$ is an $\tilde{H}$-selfadjoint matrix in the subalgebra $\Omega_{2n}$.
\end{lemma}

\begin{proof}
From the properties of the map $\omega_n$ and the fact that it is an isomorphism from $\mathbb{H}^{n\times n}$ to $\Omega_{2n}$ we see that $A^m=B$ if and only if $(\omega_n(A))^m=\omega_n(B)$, and $A$ is $H$-selfadjoint if and only if $\omega_n(A)$ is $\omega_n(H)$-selfadjoint. 
\end{proof}

Because of this lemma, necessary and sufficient conditions for the existence of an $H$-selfadjoint $m$th root of an $H$-selfadjoint matrix $B$ are the same as the necessary and sufficient conditions for the existence of an $\tilde{H}$-selfadjoint $m$th root of an $\tilde{H}$-selfadjoint matrix $\tilde{B}$ where $\tilde{B},\tilde{H}\in\Omega_{2n}$. If we could work in $\mathbb{C}^{2n\times 2n}$ we would now be done by simply referring to \cite{mthroots}. However, since $\omega_n$ is an isomorphism between $\mathbb{H}^{n\times n}$ and the subalgebra $\Omega_{2n}$ of $\mathbb{C}^{2n\times 2n}$, we have to be more careful.

We now first present in Theorem~\ref{ThmN&SconditionsOmegaB} our main theorem for the existence of $H$-selfadjoint $m$th roots of $H$-selfadjoint matrices in $\Omega_{2n}$. The proof of this theorem may be split into the following separate parts due to Lemma~\ref{LemBlockPerBlock}, viz.\ the case where $\tilde{B}$ has only positive eigenvalues, the case where $\tilde{B}$ has only nonreal eigenvalues, the case where $\tilde{B}$ has only negative eigenvalues (separated into two cases for $m$ even and for $m$ odd), and lastly, the case where $\tilde{B}$ has only zero as an eigenvalue. These we state and prove in Theorems~\ref{ThmExistencePosEig}, \ref{ThmExistenceNonreal}, \ref{ThmExistenceNegEven}, \ref{ThmExistenceNegOdd} and \ref{ThmExistenceZeroEig}.

\begin{theorem}\label{ThmN&SconditionsOmegaB}
Let $\tilde{B},\tilde{H}$ be matrices in the subalgebra $\Omega_{2n}$ such that $\tilde{H}$ is invertible and Hermitian, and $\tilde{B}$ is $\tilde{H}$-selfadjoint. Then there exists an $\tilde{H}$-selfadjoint matrix in $\Omega_{2n}$, say $\tilde{A}$,  such that $\tilde{A}^m=\tilde{B}$ if and only if the canonical form of $(\tilde{B},\tilde{H})$ has the following properties:
\begin{enumerate}
\item[\rm 1.] The part of the canonical form corresponding to negative eigenvalues when $m$ is even, say $(\tilde{B}_-,\tilde{H}_-)$, is given by
\begin{equation*}
\tilde{B}_-=\bigoplus_{j=1}^t\left(J_{k_j}(\lambda_j)\oplus J_{k_j}(\lambda_j)\right)\oplus \bigoplus_{j=1}^t\left(J_{k_j}(\lambda_j)\oplus J_{k_j}(\lambda_j)\right),\end{equation*}
\begin{equation*}
\tilde{H}_-=\bigoplus_{j=1}^t\left(Q_{k_j}\oplus -Q_{k_j}\right)\oplus \bigoplus_{j=1}^t\left(Q_{k_j}\oplus -Q_{k_j}\right),
\end{equation*}
where $\lambda_j<0$.
\item[\rm 2.] The part of the canonical form corresponding to zero eigenvalues, say $(\tilde{B}_0,\tilde{H}_0)$, is given by 
\begin{equation*} 
\tilde{B}_0=\bigoplus_{j=1}^t\left(\bigoplus_{i=1}^{r_j}J_{a_{j}+1}(0)\oplus\bigoplus_{i=r_j+1}^{m}J_{a_j}(0)\right)\oplus\bigoplus_{j=1}^t\left(\bigoplus_{i=1}^{r_j}J_{a_{j}+1}(0)\oplus\bigoplus_{i=r_j+1}^{m}J_{a_j}(0)\right)
\end{equation*}
and 
\begin{equation*}
\tilde{H}_0=\bigoplus_{j=1}^t\left(\bigoplus_{i=1}^{r_j}\varepsilon_i^{(j)}Q_{a_{j}+1}\oplus\bigoplus_{i=r_j+1}^{m}\varepsilon_{i}^{(j)}Q_{a_j}\right)\oplus\bigoplus_{j=1}^t\left(\bigoplus_{i=1}^{r_j}\varepsilon_i^{(j)}Q_{a_{j}+1}\oplus\bigoplus_{i=r_j+1}^{m}\varepsilon_{i}^{(j)}Q_{a_j}\right),
\end{equation*}
for some $a_j,r_j\in\mathbb{Z}$ with $0<r_j\leq m$. The signs are as follows, given in terms of $\eta_j$, where $\eta_j$ could be either $1$ or $-1$: If $r_j$ (respectively $m-r_j$) is even, half of $\varepsilon_i^{(j)}$ for $i=1,\ldots,r_j$ (respectively for $i=r_j+1,\ldots,m$) is equal to $\eta_j$ and the other half is equal to $-\eta_j$. If $r_j$ (respectively $m-r_j$) is odd, there is one more of $\varepsilon_i^{(j)}$ for $i=1,\ldots,r_j$ (respectively for $i=r_j+1,\ldots,m$) equal to $\eta_j$ than those equal to $-\eta_j$.
\end{enumerate}

\end{theorem}

Note that it follows from Lemma~\ref{Lem2.1recall} (which also holds for matrices in $\Omega_{2n}$) that it is sufficient to assume that the pair $(\tilde{B},\tilde{H})$ is in canonical form.

For the positive eigenvalue case, we now prove the following:
\begin{theorem}\label{ThmExistencePosEig}
Let $\tilde{B},\tilde{H}\in\Omega_{2n}$, where $\tilde{H}$ is invertible and Hermitian, and $\tilde{B}$ is $\tilde{H}$-selfadjoint with a spectrum consisting of only positive real numbers. Then there exists an $\tilde{H}$-selfadjoint matrix in $\Omega_{2n}$, say $\tilde{A}$, such that $\tilde{A}^m=\tilde{B}$.
\end{theorem}

\begin{proof}
Let $\tilde{B}\in\Omega_{2n}$ be $\tilde{H}$-selfadjoint with only positive real eigenvalues, where $\tilde{H}\in\Omega_{2n}$ is invertible and Hermitian. We can assume that $\tilde{B}=J_n(\lambda)\oplus J_n(\lambda)$, where $\lambda$ is a positive real number, and that $\tilde{H}=\varepsilon Q_n\oplus\varepsilon Q_n$, $\varepsilon=\pm1$.  To construct an $m$th root of $\tilde{B}$ which is $\tilde{H}$-selfadjoint and also in $\Omega_{2n}$, we let $J=J_n(\mu)\oplus J_n(\mu)$ where $\mu$ is the positive real $m$th root of $\lambda$. Then both $J$ and $J^m=(J_n(\mu))^m\oplus (J_n(\mu))^m$ are $\tilde{H}$-selfadjoint, and the Jordan normal form of $J^m$ is equal to $\tilde{B}$. We now wish to find an invertible matrix $P\in\Omega_{2n}$ such that equations 
\begin{equation}\label{eq2PeqPoseig}
P^{-1}J^mP=\tilde{B}\quad\textup{and}\quad P^*\tilde{H}P=\tilde{H}
\end{equation}
hold. To ensure that the first equation holds, let $P=P_1\oplus P_2$ with
\begin{equation*}
\begin{split}
P_1=\begin{bmatrix}
\left((J_n(\mu))^m-\lambda I\right)^{n-1}y & \cdots & \left((J_n(\mu))^m-\lambda I\right)y & y
\end{bmatrix},\\
P_2=\begin{bmatrix}
\left((J_n(\mu))^m-\lambda I\right)^{n-1}z & \cdots & \left((J_n(\mu))^m-\lambda I\right)z & z
\end{bmatrix},
\end{split}
\end{equation*}
where $y,z\in\textup{Ker}\left((J_n(\mu))^m-\lambda I\right)^n=\mathbb{C}^n$ but $y,z\notin\textup{Ker}\left((J_n(\mu))^m-\lambda I\right)^{n-1}$. From the choices of $y$ and $z$ it is clear that $P_1$ and $P_2$ are invertible $n\times n$ matrices and hence the $2n\times 2n$ matrix $P$ is invertible. Let $z=\bar{y}$, so that $P_2=\bar{P}_1$; then $P$ is in $\Omega_{2n}$. Note that $P^*\tilde{H}P=\tilde{H}$ if and only if $P_1^*Q_nP_1=Q_n$ and since $P_1^*Q_nP_1$ is a lower anti-triangular Hankel matrix, $P^*\tilde{H}P=\tilde{H}$ holds if and only if
\begin{equation*}
\phi_{n,j}(y)=[p_j,p_n]=y^*Q_n\left((J_n(\mu))^m-\lambda I\right)^{n-j}y=\begin{cases}
1 & \textup{if }j=1, \\
0 & \textup{if }j=2,\ldots,n,
\end{cases}
\end{equation*} where $\phi_{i,j}(y)$ denotes the entries in the matrix $P_1^*Q_nP_1$ and $p_i$ denotes the columns in the matrix $P_1$. As illustrated in \cite{mthroots}, one can easily find one solution to these equations by assuming that $y$ is real. Therefore there exists a matrix $P\in\Omega_{2n}$ satisfying the equations \eqref{eq2PeqPoseig}. Then by using Lemma~\ref{Lem2.1recall}, the matrix $\tilde{A}:=P^{-1}JP$ is an $\tilde{H}$-selfadjoint $m$th root of $\tilde{B}$, and $\tilde{A}$ is also in $\Omega_{2n}$ since $P$ is.
\end{proof}

Now, for the nonreal eigenvalue case. 
\begin{theorem}\label{ThmExistenceNonreal}
Let $\tilde{B},\tilde{H}\in\Omega_{4n}$, where $\tilde{H}$ is invertible and Hermitian, and $\tilde{B}$ is $\tilde{H}$-selfadjoint with a spectrum consisting of only nonreal numbers. Then there exists an $\tilde{H}$-selfadjoint matrix in $\Omega_{4n}$, say $\tilde{A}$, such that $\tilde{A}^m=\tilde{B}$. 
\end{theorem}

\begin{proof}
Let $\tilde{B}\in\Omega_{4n}$ be $\tilde{H}$-selfadjoint with only nonreal eigenvalues, where $\tilde{H}\in\Omega_{4n}$ is invertible and Hermitian. We can assume that $\tilde{B}=J_n(\lambda)\oplus J_n(\bar{\lambda})\oplus J_n(\bar{\lambda})\oplus J_n(\lambda)$, where $\lambda$ is a nonreal number, and that $\tilde{H}=Q_{2n}\oplus Q_{2n}$.  To construct an $m$th root of $\tilde{B}$, let $J=J_n(\mu)\oplus J_n(\bar{\mu})\oplus J_n(\bar{\mu})\oplus J_n(\mu)$ where $\mu$ is any $m$th root of $\lambda$. Then the Jordan normal form of $J^m$ is equal to $\tilde{B}$, and both $J$ and $J^m$ are $\tilde{H}$-selfadjoint. Let $P=P_1\oplus \bar{P}_1\oplus \bar{P}_1\oplus P_1\in\Omega_{4n}$ where 
\begin{equation*}
P_1=\begin{bmatrix}
\left((J_n(\mu))^m-\lambda I\right)^{n-1}y & \cdots & \left((J_n(\mu)^m-\lambda I\right)y & y
\end{bmatrix},
\end{equation*}
with $y\in\textup{Ker}\left((J_n(\mu))^m-\lambda I\right)^n=\mathbb{C}^n$ but $y\notin\textup{Ker}\left((J_n(\mu))^m-\lambda I\right)^{n-1}$. Then $P^*\tilde{H}P=\tilde{H}$ if and only if $P_1^TQ_nP_1=Q_n$, and according to the proof of Theorem~2.4 in \cite{mthroots}, there exists a solution to the latter equation. Hence, there exists an invertible matrix $P\in\Omega_{4n}$ such that the equations
\begin{equation*}
P^{-1}J^mP=\tilde{B}\quad\textup{and}\quad P^*\tilde{H}P=\tilde{H}
\end{equation*}
hold. Once again, by Lemma~\ref{Lem2.1recall}, the matrix $\tilde{A}:=P^{-1}JP$ is an $\tilde{H}$-selfadjoint $m$th root of $\tilde{B}$ and is in $\Omega_{4n}$.
\end{proof}

Before stating the results for the negative case, we give the following result which was obtained by applying $\omega_n$ to all matrices in Lemma~2.12 of \cite{mthroots}.
\begin{lemma}\label{Lem2.12recall}
Let $T\in\Omega_{2n}$ be a diagonal block matrix consisting of an upper triangular Toeplitz matrix with diagonal entries $t_1,\ldots,t_n$ and its complex conjugate. Let the diagonal entries $\lambda=t_1$ be real, $t_2$ nonzero and 
$B=T\oplus \bar{T}$. Then $B$ is $(Q_{2n}\oplus Q_{2n})$-selfadjoint, and the pair $(B, Q_{2n}\oplus Q_{2n})$ is unitarily similar to 
$$(J_n(\lambda) \oplus J_n(\lambda)\oplus J_n(\lambda) \oplus J_n(\lambda), Q_n \oplus -Q_n\oplus Q_n \oplus -Q_n).$$ 
\end{lemma}

Turning to the case of the negative eigenvalue for a matrix $\tilde{B}$, we first point out that if $m$ is even, any $m$th root $\tilde{A}$ will necessarily have only complex eigenvalues, so in order for $\tilde{A}$ to be $\tilde{H}$-selfadjoint, by Theorem~\ref{ThmcanonformOmega}, it (and so also $\tilde{B}$)  would have to be in $\Omega_{4n}$. We now first prove a result for $m$ even.
\begin{theorem}\label{ThmExistenceNegEven}
Let $\tilde{H}\in\Omega_{4n}$ be an invertible Hermitian matrix and let $\tilde{B}\in\Omega_{4n}$ be an $\tilde{H}$-selfadjoint matrix with a spectrum consisting of only negative real numbers. Then, for an even positive integer $m$, there exists an $\tilde{H}$-selfadjoint matrix in $\Omega_{4n}$, say $\tilde{A}$, such that $\tilde{A}^m=\tilde{B}$ if and only if the canonical form of $(\tilde{B},\tilde{H})$ is given by
\begin{equation}\label{eqcanonformNegB}
S^{-1}\tilde{B}S=\bigoplus_{j=1}^t\left(J_{k_j}(\lambda_j)\oplus J_{k_j}(\lambda_j)\right)\oplus \bigoplus_{j=1}^t\left(J_{k_j}(\lambda_j)\oplus J_{k_j}(\lambda_j)\right),
\end{equation}
and
\begin{equation}\label{eqcanonformNegH}
S^{*}\tilde{H}S=\bigoplus_{j=1}^t\left(Q_{k_j}\oplus -Q_{k_j}\right)\oplus \bigoplus_{j=1}^t\left(Q_{k_j}\oplus -Q_{k_j}\right),
\end{equation}
where $\lambda_j<0$ and $S$ is some invertible matrix in $\Omega_{4n}$.
\end{theorem}

\begin{proof}
Let $\tilde{B}$ be an $\tilde{H}$-selfadjoint matrix where both $\tilde{B}$ and $\tilde{H}$ are in $\Omega_{4n}$ and $\tilde{B}$ has only negative eigenvalues. Assume that there exists an $\tilde{H}$-selfadjoint $m$th root $\tilde{A}\in\Omega_{4n}$ of $\tilde{B}$, i.e.\ $\tilde{A}^m=\tilde{B}$. Denote the eigenvalues of $\tilde{B}$ by $\lambda_j$, and let $\mu_j$ be any $m$th root of $\lambda_j$. Since $m$ is even, $\mu_j$ is nonreal. Thus from Theorem~\ref{ThmcanonformOmega} we know that the canonical form of $(\tilde{A},\tilde{H})$ is $(J,Q)$ where
\begin{equation*}
J=\bigoplus_{j=1}^t\left(J_{k_j}(\mu_j)\oplus J_{k_j}(\bar{\mu}_j)\right)\oplus\bigoplus_{j=1}^t \left(J_{k_j}(\bar{\mu}_j)\oplus J_{k_j}(\mu_j)\right)
\end{equation*}
and
\begin{equation*}
Q=\bigoplus_{j=1}^tQ_{2k_j}\oplus\bigoplus_{j=1}^t Q_{2k_j}.
\end{equation*}
for some $t$. Hence there exists an invertible matrix $P\in\Omega_{4n}$ such that the equations $P^{-1}\tilde{A}P=J$ and $P^*\tilde{H}P=Q$ hold. Consider
\begin{equation*}
P^{-1}\tilde{B}P=(P^{-1}\tilde{A}P)^m=J^m=\bigoplus_{j=1}^t\left((J_{k_j}(\mu_j))^m\oplus (J_{k_j}(\bar{\mu}_j))^m\right)\oplus\bigoplus_{j=1}^t\left( (J_{k_j}(\bar{\mu}_j))^m\oplus (J_{k_j}(\mu_j))^m\right),
\end{equation*}
 and note that this matrix is $P^*\tilde{H}P$-selfadjoint. Next, by taking 
 $$T_j=(J_{k_j}(\mu_j))^m\oplus (J_{k_j}(\bar{\mu}_j))^m$$ 
for $j=1,\ldots,t$, and applying Lemma~\ref{Lem2.12recall} we have  that $(T_j\oplus \bar{T}_j,\;Q_{2k_j}\oplus Q_{2k_j})$ is unitarily similar to 
 \begin{equation*}
 \left( J_{k_j}(\lambda_j)\oplus J_{k_j}(\lambda_j)\oplus J_{k_j}(\lambda_j)\oplus J_{k_j}(\lambda_j),\;Q_{k_j}\oplus -Q_{k_j}\oplus Q_{k_j}\oplus -Q_{k_j}\right)
 \end{equation*}
for all $j=1,\ldots,t$. Therefore by using the permutation matrix defined in \eqref{eqPermutationMatrix}, we see that the canonical form of $(\tilde{B},\tilde{H})$ is given by \eqref{eqcanonformNegB} and \eqref{eqcanonformNegH}.

Conversely, let $\tilde{B}$ be $\tilde{H}$-selfadjoint and such that the canonical form of $(\tilde{B},\tilde{H})$ is given by \eqref{eqcanonformNegB} and \eqref{eqcanonformNegH}, and let $m$ be even. Assume that 
\begin{equation*}
\tilde{B}=J_n(\lambda)\oplus J_n(\lambda)\oplus J_n(\lambda)\oplus J_n(\lambda)\quad\textup{and}\quad \tilde{H}=Q_n\oplus -Q_n\oplus Q_n\oplus -Q_n,
\end{equation*}
where $\lambda<0$. Then let $J=J_n(\mu)\oplus J_n(\bar{\mu})\oplus J_n(\bar{\mu})\oplus J_n(\mu)$ where $\mu$ is any $m$th root of $\lambda$. The number $\mu$ is nonreal and therefore $J$ is $Q$-selfadjoint where $Q=Q_{2n}\oplus Q_{2n}$. Note that the matrix $J^m$ has $\lambda$ on its main diagonal and satisfies the conditions of $T\oplus \bar{T}$ as in  Lemma~\ref{Lem2.12recall}. Thus it follows that the pair $(J^m,Q_{2n}\oplus Q_{2n})$ is unitarily similar to the pair $(\tilde{B},\tilde{H})$, i.e., there exists an invertible matrix $P\in\Omega_{4n}$ such that the equations $P^{-1}J^mP=\tilde{B}$ and $P^*QP=\tilde{H}$ hold. Finally, from Lemma~\ref{Lem2.1recall} the matrix $\tilde{A}:=P^{-1}JP\in\Omega_{4n}$ is an $\tilde{H}$-selfadjoint $m$th root of $\tilde{B}$.
\end{proof}

Next we give the negative eigenvalue case where $m$ is odd. This case is similar to the positive eigenvalue case.
\begin{theorem}\label{ThmExistenceNegOdd}
Let $\tilde{H}\in\Omega_{2n}$ be an invertible Hermitian matrix and $\tilde{B}\in\Omega_{2n}$ an $\tilde{H}$-selfadjoint matrix with a spectrum consisting of only negative real numbers. Then, for $m$ odd, there exists an $\tilde{H}$-selfadjoint matrix in $\Omega_{2n}$, say $\tilde{A}$, such that $\tilde{A}^m=\tilde{B}$. 
\end{theorem}

\begin{proof}
Let $\tilde{B}\in\Omega_{2n}$ be $\tilde{H}$-selfadjoint with only negative real eigenvalues, where $\tilde{H}\in\Omega_{2n}$ is invertible and Hermitian, and let $m$ be odd. Assume that $\tilde{B}=J_n(\lambda)\oplus J_n(\lambda)$ and $\tilde{H}=\varepsilon Q_n\oplus\varepsilon Q_n$, where $\lambda<0$ and $\varepsilon=\pm 1$. Since $m$ is odd, we can take $\mu$ to be the real $m$th root of $\lambda$ and let $J=J_n(\mu)\oplus J_n(\mu)$. Then the Jordan normal form of $J^m$ is equal to $\tilde{B}$, and both $J$ and $J^m$ are $\tilde{H}$-selfadjoint. Similarly to Theorem~\ref{ThmExistencePosEig}, we can construct an invertible matrix $P\in\Omega_{2n}$ such that the equations
\begin{equation*}
P^{-1}J^mP=\tilde{B}\quad\textup{and}\quad P^*\tilde{H}P=\tilde{H}
\end{equation*}
hold. Therefore, from Lemma~\ref{Lem2.1recall}, the matrix $\tilde{A}:=P^{-1}JP\in\Omega_{2n}$ is an $\tilde{H}$-selfadjoint $m$th root of $\tilde{B}$.
\end{proof}

In the case where the matrix $\tilde{B}$ has only the number zero as an eigenvalue, we instantly notice a necessary condition for the existence of an $m$th root $\tilde{A}$ of $\tilde{B}$. Note that $\tilde{A}$ is not necessarily $\tilde{H}$-selfadjoint. 
Since each number in the Segre characteristic of  $\tilde{A}$  corresponding to the zero eigenvalue occurs twice (see Corollary~\ref{CorDoubleNrinSegre}) and from the way $m$th roots are formed for nilpotent matrices, see for example \cite{BorweinR,mthroots}, the $m$-tuples in the Segre characteristic (or some reordering thereof) corresponding to the zero eigenvalue of $\tilde{B}$ have to exist in pairs, i.e.\ there are two of each $m$-tuple. For example, with $m=4$, if the nonzero part of the Segre characteristic of a nilpotent matrix $\tilde{B}\in\Omega_{20}$ is $(3,3,2,2),(3,3,2,2)$, then the nonzero part of the Segre characteristic of any $m$th root $\tilde{A}\in\Omega_{20}$ is $(10,10)$. 

Here we need another result from \cite{mthroots} which was obtained by applying $\omega_n$ to all matrices:
\begin{lemma}\label{Lem2.6recall}
Let $A$ be equal to $J_n(0)\oplus J_n(0)$. Then $A^m$ has Jordan normal form 
\begin{equation*}
\bigoplus_{i=1}^r J_{a+1}(0)\oplus\bigoplus_{i=1}^{m-r} J_a(0)\oplus\bigoplus_{i=1}^r J_{a+1}(0)\oplus\bigoplus_{i=1}^{m-r} J_a(0),
\end{equation*}
where $n=am+r$, for $a,r\in\mathbb{Z}$, $0<r\leq m$.
\end{lemma}

We now give the result for this last case.

\begin{theorem}\label{ThmExistenceZeroEig}
Let $\tilde{H}\in\Omega_{2n}$ be an invertible Hermitian matrix and $\tilde{B}\in\Omega_{2n}$ an $\tilde{H}$-selfadjoint matrix with a spectrum consisting of only the number zero. Then there exists an $\tilde{H}$-selfadjoint matrix in $\Omega_{2n}$, say $\tilde{A}$, such that $\tilde{A}^m=\tilde{B}$ if and only if the following properties hold:
\begin{enumerate}
\item[\rm 1.] There exists a reordering of the Segre characteristic of $\tilde{B}$ such that each $m$-tuple occurs twice and the difference between any two numbers in each $m$-tuple is at most one.
\item[\rm 2.] By using a reordering satisfying the first property, the canonical form of $(\tilde{B},\tilde{H})$ is given by $(J_B\oplus J_B,H_B\oplus H_B)$ where
\begin{equation*}
J_B=\bigoplus_{j=1}^t\left(\bigoplus_{i=1}^{r_j}J_{a_{j}+1}(0)\oplus\bigoplus_{i=r_j+1}^{m}J_{a_j}(0)\right),
\end{equation*}
and
\begin{equation}\label{eqZeroCanonform}
H_B=\bigoplus_{j=1}^t\left(\bigoplus_{i=1}^{r_j}\varepsilon_i^{(j)}Q_{a_{j}+1}\oplus\bigoplus_{i=r_j+1}^{m}\varepsilon_{i}^{(j)}Q_{a_j}\right),
\end{equation}
for some $a_j,r_j\in\mathbb{Z}$ with $0<r_j\leq m$, and the signs are as follows, given in terms of $\eta_j$, where $\eta_j$ could be either $1$ or $-1$: If $r_j$ (respectively $m-r_j$) is even, half of the $\varepsilon_i^{(j)}$ for $i=1,\ldots,r_j$ (respectively for $i=r_j+1,\ldots,m$) are equal to $\eta_j$ and the other half are equal to $-\eta_j$. If $r_j$ (respectively $m-r_j$) is odd, there is one more of the $\varepsilon_i^{(j)}$ for $i=1,\ldots,r_j$ (respectively for $i=r_j+1,\ldots,m$) equal to $\eta_j$ than to $-\eta_j$.
\end{enumerate} 
\end{theorem}

\begin{proof}
Let $\tilde{B}\in\Omega_{2n}$ be $\tilde{H}$-selfadjoint with only zero in its spectrum, where $\tilde{H}\in\Omega_{2n}$ is invertible and Hermitian. Assume there exists an $\tilde{H}$-selfadjoint $m$th root $\tilde{A}\in\Omega_{2n}$, i.e.\ $\tilde{A}^m=\tilde{B}$. From Theorem~\ref{ThmcanonformOmega} there exists an invertible matrix $S\in\Omega_{2n}$ such that the canonical form of $(\tilde{A},\tilde{H})$ is given by 
\begin{equation*}
S^{-1}\tilde{A}S=\bigoplus_{j=1}^tJ_{k_j}(0)\oplus \bigoplus_{j=1}^tJ_{k_j}(0)\quad\textup{and}\quad S^*\tilde{H}S=\bigoplus_{j=1}^t\eta_jQ_{k_j}\oplus \bigoplus_{j=1}^t\eta_jQ_{k_j},
\end{equation*}
for some $t$, $k_j$ and signs $\eta_j=\pm1$. Consider
\begin{equation*}
S^{-1}\tilde{B}S=(S^{-1}\tilde{A}S)^m=\bigoplus_{j=1}^t(J_{k_j}(0))^m\oplus \bigoplus_{j=1}^t(J_{k_j}(0))^m.
\end{equation*}
Therefore by using Lemma~\ref{Lem2.6recall} and the permutation matrix defined in \eqref{eqPermutationMatrix}, the matrix $\tilde{B}$ has Jordan normal form
\begin{equation}\label{eqZeroproof1}
\bigoplus_{j=1}^t\left(\bigoplus_{i=1}^{r_j}J_{a_{j}+1}(0)\oplus\bigoplus_{i=1}^{m-r_j}J_{a_j}(0)\right)\oplus\bigoplus_{j=1}^t\left(\bigoplus_{i=1}^{r_j}J_{a_{j}+1}(0)\oplus\bigoplus_{i=1}^{m-r_j}J_{a_j}(0)\right),
\end{equation}
where $k_j=a_jm+r_j$, $a_j,r_j\in\mathbb{Z}$ and $0<r_j\leq m$. Hence from \eqref{eqZeroproof1} we see that there exists a reordering of the Segre characteristic of $\tilde{B}$ in which each $m$-tuple occurs twice and where the difference of any two numbers in each $m$-tuple is at most one.

Since the Jordan normal form of $\tilde{B}$ is given in \eqref{eqZeroproof1}, by Theorem~\ref{ThmcanonformOmega} the corresponding matrix in the canonical form is given by $H_B\oplus H_B$ where
\begin{equation*}
H_B=\bigoplus_{j=1}^t\left(\bigoplus_{i=1}^{r_j}\varepsilon_i^{(j)}Q_{a_{j}+1}\oplus\bigoplus_{i=r_j+1}^{m}\varepsilon_{i}^{(j)}Q_{a_j}\right),
\end{equation*}
for $\varepsilon_i^{(j)}=\pm1$, $i=1,\ldots,m$, $j=1,\ldots,t$. From the assumption that an $\tilde{H}$-selfadjoint $m$th root exists, we know that the two properties given in \cite[Theorem~2.5]{mthroots} hold.  The second property is used to find the signs of $H_B$. Thus it follows: if $r_j$ (respectively $m-r_j$) is even, half of the $\varepsilon_i^{(j)}$ for $i=1,\ldots,r_j$ (respectively for $i=r_j+1,\ldots,m$) are equal to $\eta_j$ and the other half are equal to $-\eta_j$. If $r_j$ (respectively $m-r_j$) is odd, there is one more of the $\varepsilon_i^{(j)}$ for $i=1,\ldots,r_j$ (respectively for $i=r_j+1,\ldots,m$) equal to $\eta_j$ than  to $-\eta_j$.

Conversely, let $\tilde{B}$ be an $\tilde{H}$-selfadjoint matrix and suppose that the two properties in the theorem hold. Assume from the first property that $\tilde{B}=B_1\oplus B_1\in\Omega_{2n}$ where
\begin{equation*}
B_1=\bigoplus_{j=1}^t\left(\bigoplus_{i=1}^{r_j}J_{a_{j}+1}(0)\oplus\bigoplus_{i=1}^{m-r_j}J_{a_j}(0)\right),
\end{equation*}
and assume that $\tilde{H}=H_1\oplus H_1$ where $H_1=H_B$ is as in \eqref{eqZeroCanonform}. Let 
\begin{equation*}
J=\bigoplus_{j=1}^tJ_{t_j}(0)\oplus \bigoplus_{j=1}^tJ_{t_j}(0),
\end{equation*}
where $t_j$ is the sum of the sizes of the blocks in $\tilde{B}$ which correspond to one $m$-tuple, i.e.\ $t_j=r_j(a_j+1)+(m-r_j)(a_j)=a_jm+r_j$. Thus by using Lemma~\ref{Lem2.6recall} and the permutation matrix defined by \eqref{eqPermutationMatrix}, the Jordan normal form of $J^m$ is equal to $\tilde{B}$. Also, let $Q=\bigoplus_{j=1}^t\varepsilon_{j} Q_{t_j}\oplus \bigoplus_{j=1}^t\varepsilon_j Q_{t_j}$ where $\varepsilon_j=\eta_j$ is obtained from the signs of $H_B$, then $J$ is $Q$-selfadjoint. According to  \cite[Theorem~2.5]{mthroots} there exists an invertible matrix $P_1$ such that
\begin{equation*}
P_1^{-1}\Big(\bigoplus_{j=1}^tJ_{t_j}(0)\Big)^mP_1=B_1 \quad\textup{and}\quad P_1^*\Big(\bigoplus_{j=1}^t\varepsilon_jQ_{t_j}\Big)P_1=H_1.
\end{equation*}
Let $P=P_1\oplus \bar{P}_1$, then $P$ is an invertible matrix in $\Omega_{2n}$ and the equations
\begin{equation*}
P^{-1}J^mP=\tilde{B}\quad \textup{and}\quad P^*QP=\tilde{H}
\end{equation*}
hold. From Lemma~\ref{Lem2.1recall} the matrix $\tilde{A}:=P^{-1}JP\in\Omega_{2n}$ is an $\tilde{H}$-selfadjoint $m$th root of $\tilde{B}$.
\end{proof}

\bigskip
{\bf Acknowledgments.} This work is based on research supported in part by the DSI-NRF Centre of Excellence in Mathematical and Statistical Sciences (CoE-MaSS). Opinions expressed and conclusions arrived at are those of the authors and are not necessarily to be attributed to the CoE-MaSS.


\end{document}